\documentclass{amsart}[12pt]
\usepackage{amsmath}
\usepackage{amssymb}
\usepackage{amsthm}
\usepackage{verbatim}
\usepackage[pdftex]{graphics}

\newtheorem{theorem}{Theorem}[section]
\newtheorem{lemma}[theorem]{Lemma}
\newtheorem{corollary}[theorem]{Corollary}
\newtheorem{proposition}[theorem]{Proposition}

\newtheorem{example}[theorem]{Example}
\newtheorem{problem}[theorem]{Problem}

\title[Localizing algebras and invariant subspaces]{ Localizing algebras and invariant subspaces}
\author{Miguel Lacruz and Luis Rodr\'{\i}guez-Piazza}
\address{Departamento de An\'{a}lisis Matem\'{a}tico, Facultad de Matem\'{a}ticas, Universidad de Sevilla, Apartado 1160, 41080  Sevilla, Spain}
\email{lacruz@us.es}
\address{Departamento de An\'{a}lisis Matem\'{a}tico, Facultad de Matem\'{a}ticas, Universidad de Sevilla, Apartado 1160, 41080 Sevilla, Spain}
\email{piazza@us.es}

\begin{document}
\renewcommand{\thefootnote}{}
\begin{abstract}
It is shown that the algebra \(L^\infty(\mu)\) of all bounded measurable functions with respect to a finite measure \(\mu\) is localizing on the Hilbert space \(L^2(\mu)\)  if and only if the  measure \(\mu\) has an atom.  Next, it is shown that the algebra \(H^\infty({\mathbb D})\) of all bounded analytic multipliers on the unit disc fails to be localizing, both on the Bergman  space \(A^2({\mathbb D})\) and on  the Hardy space \(H^2({\mathbb D}).\)  Then, several conditions are provided for the algebra generated by a diagonal operator on a Hilbert space to be localizing. Finally, a theorem is provided about the existence  of hyperinvariant subspaces for  operators with  a localizing subspace of extended eigenoperators. This theorem extends and unifies some previously known results of Scott Brown and Kim, Moore and Pearcy, and Lomonosov, Radjavi and Troitsky.
\end{abstract}
\date{10 May 2013}
%47L10 Algebras of operators on Banach spaces and other topological linear spaces
%47A15 Invariant subspaces
\maketitle
\section{Introduction}
\noindent
Let  \(\mathcal{B}(E)\) denote the algebra of all bounded  linear operators defined on a Banach space \(E.\) A subspace \({\mathcal X} \subseteq {\mathcal B}(E)\) is said to be {\em localizing} provided that there is a closed ball \(B \subseteq E\) such that \( 0 \notin B\) and such  that for every sequence \((x_n)\) in \(B\) there is a subsequence \((x_{n_j})\) and a sequence \((X_j)\) in \({\mathcal X}\) such that \(\|X_j\| \leq 1\) and such that the sequence \((X_jx_{n_j})\) converges in norm to some nonzero vector. This notion  was introduced by Lomonosov, Radjavi, and Troitsky \cite{LRT}  as a side condition to build  invariant subspaces for bounded linear operators on Banach spaces.  

Recall that the commutant of an operator \(T \in {\mathcal B}(E)\) is the subalgebra \(\{T\}^\prime \subseteq {\mathcal B}(E)\) of all operators that commute with \(T.\) 
A subspace \(F \subseteq E\) is said to be invariant under an operator \(T \in {\mathcal B}(E)\) provided that \(TF \subseteq F\). A subspace \(F \subseteq E\) is said to be  invariant under a subalgebra \({\mathcal R} \subseteq {\mathcal B}(E)\) if \(F\) is invariant under  every \(R \in {\mathcal R} \). A subspace \(F \subseteq E\) is said to be hyperinvariant under an operator \(T \in {\mathcal B}(E)\) if \(F\) is invariant under  the commutant \(\{T\}'\). A subalgebra \({\mathcal R} \subseteq {\mathcal B}(E)\) is said to be transitive if the only closed subspaces invariant under \({\mathcal R}\) are the trivial ones,  namely, \(F= \{0\}\) and \(F=E.\) It is easy to see that this is equivalent to saying that for every \(x \in E \backslash \{0\},\) the subspace \(\{Rx \colon R \in {\mathcal R}\}\) is dense in \(E.\)

We shall denote by \({\rm ball}({\mathcal X})\) the unit ball of a subspace \(\mathcal{X} \subseteq \mathcal{B}(E).\) Also, we shall denote by \(\sigma\) the weak operator topology on  \( \mathcal{B}(E).\) Recall that for a convex subset of \({\mathcal B}(E),\)  the closure in the weak operator topology agrees with the closure in the strong operator topology.

Lomonosov, Radjavi, and Troitsky \cite{LRT}  proved among other results the following
\begin{theorem} If \(T\) is a nonzero quasinilpotent operator on a Banach space and its commutant \(\{T \}^\prime\) is a localizing algebra, then \(T\)  has a nontrivial hyperinvariant subspace.
\label{quasi}
\end{theorem}
\noindent
They also made the observation that any operator algebra containing a nonzero compact operator is a localizing algebra. This observation can be refined as follows.

\begin{proposition}
\label{compact}
Let \(\mathcal{X}\) be a subspace  of \(\mathcal{B}(E)\) such  that the closure  of its unit ball  in the weak operator topology contains a nonzero compact operator. Then  \(\mathcal{X}\) is localizing.
\end{proposition}

\begin{proof}
Let \(T \in \overline{\rm{ball}(\mathcal{X})}^\sigma\)be a nonzero compact operator. We may assume without loss of generality that \(\|T\|=1\). Let \(x_0 \in E\) be a vector with the property  that \(\|x_0\|=1\) and \(\|Tx_0\| \geq 	3/4\). Consider the closed ball \(B=\{x \in E: \|x-x_0\| \leq 1/4\}\). It is clear that \( 0 \notin B\). Now, let \((x_n)\) be any sequence in \(B\). Since \(T\) is compact,  there is a subsequence \((x_{n_j})\) such that \((Tx_{n_j})\) converges in norm to \(y \in E\), say. Since the closure in the weak operator topology of a convex set agrees with the closure in the strong operator topology, for every \(j \geq 1\), there is an operator \(X_j \in \rm{ball}(\mathcal{X})\) such that \(\|Tx_{n_j}-X_jx_{n_j}\| \leq 1/j\). It follows that \(\|y - X_jx_{n_j}\| \rightarrow 0\) as \(j \rightarrow \infty\). Finally, we show that \(y \neq 0\). Notice that  \(\|Tx_n - Tx_0\| \leq \|x_n -x_0\| \leq 1/4\), and since \(\|Tx_0\| \geq 3/4\), we conclude that \(\|Tx_n\| \geq 1/2\) for all \(n \geq 1,\)  so that \(\|y\| \geq 1/2\), as we wanted.
\end{proof}

The first author \cite{lacruz} provided an example of a weakly closed, localizing algebra of bounded operators on the Banach space \(C[0,1]\) that does not contain any  nonzero compact operators. As a matter of fact, the example is the algebra of all multiplication operators by continuous functions on the unit interval, and as it turns out, this is the uniformly closed unital algebra generated by the position operator. It is natural to ask  if such a construction can be carried out on a Hilbert space. This question can be formulated more precisely  as follows.

\begin{problem}
\label{genprob}
{\em Is there a Hilbert space \(H\) and a localizing algebra \({\mathcal R} \subseteq {\mathcal B}(H)\) such that  \(\overline{{\rm ball}({\mathcal R})}^\sigma\) does not contain any nonzero compact operators?}
\end{problem}

The first part of this paper (sections 2, 3, and 4) initiates the investigation  of some properties of localizing algebras bearing this question in mind, although we have not been able to solve it.

The  notions of extended eigenvalue and  eigenoperator  became popular back in the 1970s when searching for invariant subspaces of operators on Banach spaces.  A  complex scalar \(\lambda \in \mathbb{C}\) is said to be an {\em extended eigenvalue} for an operator \(T \in {\mathcal B}(E)\)  if there exists a  nonzero operator  \(X \in {\mathcal B}(E)\) such that \(TX=\lambda XT.\) Such an operator \(X\) is called an {\em extended eigenoperator} for \(T\) associated with \(\lambda.\) The following extension  of Lomonosov's  invariant subspace theorem  \cite{lomonosov} was obtained by Scott Brown \cite{brown}, and independently by  Kim, Moore and Pearcy \cite{KMP}.

\begin{theorem} 
\label{BKMP}
If a nonscalar operator \(T \in {\mathcal B}(E)\) has a compact eigenoperator then \(T\) has a nontrivial hyperinvariant subspace.
\end{theorem}
The special case that \(T\) commutes with a nonzero compact operator is the original result of Lomonosov.

Recently, the concepts of extended eigenvalue and  eigenoperator have received a considerable amount of attention, both in the context of invariant subspaces \cite{lambert} and in the study of extended eigenvalues and  extended eigenoperators for some special classes of operators \cite{BP,BLP, BS, karaev,petrovic,shkarin}.
The second part of this paper (section 5) provides a result that extends and unifies   Theorem \ref{quasi} and Theorem \ref{BKMP}. Our result can be stated as follows.

\begin{theorem} 
\label{main}
Let \(T \in {\mathcal B}(E)\) be a nonzero operator, let \(\lambda \in {\mathbb C}\)  be an extended eigenvalue of \(T\)  such that the subspace \({\mathcal X}\) of all associated extended eigenoperators is localizing and suppose that either
\begin{enumerate}
\item \(|\lambda | < 1,\)
\item \(|\lambda | > 1,\) or
\item \(|\lambda |=1\) and \(T\) is quasinilpotent.
\end{enumerate}
Then  \(T\) has a nontrivial hyperinvariant subspace.
\end{theorem}

Since the commutant  \(\{T\}^\prime\) is the family of all extended eigenoperators associated with the extended eigenvalue \(\lambda=1,\) it  follows that Theorem \ref{main} is  an extension of Theorem \ref{quasi} to the case of extended eigenvalues with \(|\lambda | = 1.\)  On the other hand, it follows from Proposition \ref{compact} that Theorem \ref{main} is an extension of Theorem \ref{BKMP} at least for  extended eigenvalues with \(|\lambda | \neq 1.\) 

The paper is organized as follows. In section \ref{asla}, it is shown that if the algebra \(L^\infty(\mu)\) of all bounded measurable functions with respect to a finite measure \(\mu\) is localizing on the Hilbert space \(L^2(\mu)\)  then the  measure \(\mu\) must have an atom.  In section \ref{bam}, it is shown  that the algebra \(H^\infty(\mathbb{D})\) of  all bounded analytic multipliers on the unit disc fails to be localizing, both on the Bergman space \(A^2(\mathbb{D})\) \nopagebreak and on the Hardy space \(H^2(\mathbb{D}).\) In section \ref{diag}, some conditions are given for the algebra generated by a diagonal operator on a Hilbert space to be localizing. In section \ref{eigen},  a proof of Theorem \ref{main} is provided, and an example is given to illustrate  that Theorem \ref{main} is more general  than Theorem \ref{BKMP}.

\section{Abelian selfadjoint localizing algebras} 
\label{asla}
\noindent
The question arises of whether  the  closure in the weak operator topology of the unit ball of a localizing algebra  of operators on a Hilbert space  must contain a nonzero compact operator.
First,  we consider the case of an abelian selfadjoint algebra. Once again, recall that the closure in the weak operator topology of a convex set agrees with the closure in the strong operator topology. Kaplansky's density theorem \cite{kaplansky}  is the assertion  that if \(\mathcal{R}\) is a selfadjoint  algebra of operators on a Hilbert space  then the strong closure of the unit ball of \(\mathcal{R}\) is the unit ball of the strong closure of \(\mathcal{R}.\) See the book of Takesaki  \cite[Theorem 4.8]{takesaki} for another reference on Kaplansky's density theorem. This result is the key to the following
\begin{proposition}
\label{kaplansky}
If \(\mathcal{R}\) is a selfadjoint algebra of operators on a Hilbert space \(H\) such that its closure in the weak operator topology   is a localizing algebra then \(\mathcal{R}\) itself is a localizing algebra.
\end{proposition}

\begin{proof}
Suppose that the closure in the weak operator topology \(\overline{\mathcal{R}}^\sigma\) is a localizing algebra and let \(B\) be a closed ball as in the definition.  Take a sequence of vectors \((x_n)\) in \(B\), extract a subsequence \((x_{n_j})\), and find a sequence of operators \((T_j)\) in  \(\rm{ball}(\overline{\mathcal{R}}^\sigma)\) and a nonzero vector \(y \in H\) such that \(\|y-T_jx_{n_j} \| \rightarrow 0\) as \(j \rightarrow \infty\).  It follows from Kaplansky's  density theorem that for every \(j \geq 1\) there is an operator \(R_j \in  \rm{ball}(\mathcal{R})\) such that \(\|T_jx_{n_j}-R_jx_{n_j}\| \leq 1/j\). Thus, \(\|y-R_jx_{n_j} \| \rightarrow 0\) as \(j \rightarrow \infty,\) as we wanted.
\end{proof}

\noindent
Let  \((\Omega,  \Sigma, \mu)\) be a finite measure space. We identify every  function \(\varphi\) in \(L^\infty(\mu)\) with the multiplication operator
\(M_\varphi\)  defined on \(L^2(\mu)\) by the expression \((M_\varphi f)(\omega)=\varphi(\omega) f(\omega) \), and  we regard \(L^\infty(\mu)\)  as a subalgebra of \({\mathcal B}(L^2(\mu))\) under this
identification. 

\begin{theorem}
\label{atomic}
If the measure \(\mu\) contains no atoms then the  algebra \(L^\infty(\mu)\) is not localizing on  the Hilbert space \(L^2(\mu)\).
\end{theorem}
\begin{proof}
Let \(f_0 \in L^2(\mu)\) and  consider a closed ball \(B= \{f \in L^2(\mu): \|f-f_0\|_2 \leq \varepsilon\}\). We must show that the algebra \(L^\infty(\mu)\) and the ball \(B\) do not satisfy the condition in the definition of a localizing algebra. First of all, there is a \(\delta >0\) such that, for each measurable subset \(A \subseteq \Omega\), the condition \(\mu(A) < \delta \) implies \(\|f_0 \cdot \chi_A\|_2 < \varepsilon\), or equivalently, \(\|f_0\cdot \chi_{A^c}-f_0\|_2<\varepsilon\). Since \(\mu\) has no atoms, we may construct a sequence \((A_n)\) of   independent, measurable subsets of \(\Omega\) such that \(\mu(A_n)=\delta/2\) for each \(n \geq 1\). Then, we define a sequence of functions \((f_n)\) inside the ball \(B\) by the expression \(f_n=f_0 \cdot \chi_{A_n^c}\). Suppose that  there is a subsequence \((f_{n_j})\), a function \(f \in L^2(\mu)\), and a sequence of functions \((\varphi_j)\) in \(L^\infty(\mu)\) such that \(\|\varphi_j\|_\infty \leq 1\) and \(\|f-\varphi_jf_{n_j}\|_2 \rightarrow 0\) as \(j \rightarrow \infty\). 
Thus, it suffices to show that \(f=0\). Now,  extract a further subsequence \((\varphi_{j_k}f_{n_{j_k}})\) that converges almost everywhere to \(f\). Next, apply the Borel-Cantelli Lemma to the sequence \((A_{n_{j_k}})\), and conclude that there is a measurable subset \(Z \subseteq \Omega\) such that \(\mu(Z)=0\) and such that  the set \(\{k \geq 1: \omega \in A_{n_{j_k}}\}\) is infinite for every \(\omega \in Z^c\). It follows  that \(f\) vanishes almost everywhere, as we wanted.
\end{proof}

\noindent
Since any maximal abelian, selfadjoint algebra can be represented as a function algebra \(L^\infty(\mu),\) as a consequence of Theorem \ref{atomic} we get the following

\begin{corollary}
\label{masa}
Let \(H\) be an infinite dimensional separable Hilbert space and let  \(\mathcal{R}\) be a maximal abelian, selfadjoint subalgebra of \(\mathcal{B}(H)\). The following conditions are equivalent:
\begin{enumerate}
\item \(\mathcal{R}\) is a localizing algebra,
\item \(\mathcal{R}\) contains a rank one operator,
\item \(\mathcal{R}\) contains a nonzero finite rank operator,
\item \(\mathcal{R}\) contains a nonzero compact operator.
\end{enumerate}
\end{corollary}

\noindent
See the book of Radjavi and Rosenthal \cite[Corollary 7.14] {RR} for a reference on the representation of maximal abelian selfadjoint algebras. We finish this section with an example of a probability measure \(\mu\) and a subalgebra \(\mathcal{R} \subseteq L^\infty(\mu)\) that fails to be localizing although its closure \(\overline{\mathcal{R}}^\sigma\) in the weak operator topology is a localizing algebra. This example goes to show that the assumption that the algebra \({\mathcal R}\) is selfadjoint cannot be dropped from the hypotheses of Proposition \ref{kaplansky}.
\vskip1ex
\noindent
\begin{example} 
\label{nsa}
{\em Let \((z_k)\) be a sequence complex scalars in the open unit disc \(\mathbb{D}\) such that \(\overline{\{z_k \colon k \geq 1\}} \supseteq \partial \mathbb{D}\) and  \(\overline{\{z_k \colon k \geq 1\}} \cap [0,1/2]= \emptyset.\) Then, let \(\lambda\) denote the  Lebesgue measure on the real line and consider the probability measure
\[
\mu = \lambda |_{[0,1/2]} + \sum_{k=1}^\infty 2^{-k-1} \delta_{z_k}.
\]
Finally, consider the subalgebra \({\mathcal R} \subseteq L^\infty(\mu)\) defined as \({\mathcal R}= \{ p(M_z) \colon p \text{ is a polynomial}\}.\)  We start with a result that is an immediate consequence of the maximum modulus principle.}
\end{example}
\begin{lemma} 
\(\|p\|_{L^\infty(\mu)}= \sup \{ |p(z)| \colon z \in \partial \mathbb{D}\}\)  for every polynomial \(p.\) 
\end{lemma}
\noindent
The next result is a  condition for an operator \(T \in \mathcal{B}(L^2(\mu))\) to belong to the weak closure of \({\rm ball}(\mathcal{R}).\)
\begin{lemma} \(T \in \overline{{\rm ball}({\mathcal R})}^\sigma\) if and only if there is some \(\varphi \in {\rm ball}(H^\infty(\mathbb{D}))\) such that \(T=M_\varphi.\)
\end{lemma}
\begin{proof} Notice that \(L^\infty(\mu)\) is  closed in the weak operator topology and that the  weak operator topology restricted to \(L^\infty(\mu)\) agrees with weak-\(\ast\) topology \(\sigma(L^\infty,L^1).\) Thus, given an operator  \(T \in \overline{{\rm ball}({\mathcal R})}^\sigma,\) there is a function \(\psi \in L^\infty (\mu)\) such that \(T=M_\psi,\) and there exists a sequence of polynomials \((p_n)\) such that \(\|p_n\|_\infty \leq 1\) and \(p_n \to \psi\) in the weak-\(\ast\) topology. Then, it follows from Montel's theorem that there is a subsequence \((p_{n_j})\) and there is a function \(\varphi \in {\rm ball}(H^\infty(\mathbb{D}))\) such that \(p_{n_j} \to \varphi\) uniformly on compact subsets of \(\mathbb{D}.\) Hence, \(p_{n_j} \to \varphi\) almost everywhere, and it follows from the bounded convergence theorem that \(p_{n_j} \to \varphi\) in the weak-\(\ast\) topology. Therefore, \(\varphi=\psi\) and \(T=M_\varphi,\) as we wanted. Conversely, suppose that there is some \(\varphi \in {\rm ball}(H^\infty(\mathbb{D}))\) such that \(T=M_\varphi,\) and let \(\varphi_r(z)=\varphi(rz)\) for \(0<r<1,\) so that \(\varphi_r(z) \to \varphi (z)\) as \(r \to 1^-\) for all \(z \in \mathbb{D},\) and it follows from the bounded convergence theorem that \(\varphi_r \to \varphi\) in the weak-\(\ast\) topology.  Since \(\varphi_r \in {\rm ball}(A({\mathbb D})),\) we have \(M_{\varphi_r} \in \overline{{\rm ball}({\mathcal R})}^{\| \cdot \|},\) so that  \(M_\varphi \in \overline{{\rm ball}({\mathcal R})}^\sigma.\) 
\end{proof}
\begin{theorem} \label{blaschke} If the sequence \((z_k)\) satisfies  the Blaschke condition 
\[
\displaystyle{\sum_{k=1}^\infty (1-|z_k|) < \infty}
\]
then the algebra \(\overline{{\mathcal R}}^\sigma\)  contains a rank one operator and therefore it is localizing.
\end{theorem}
\begin{proof} Start with  a \(B \in {\rm ball}(H^\infty(\mathbb{D}))\) such that \(B(z_k)=0\) for all \(k \geq 1\) and \(B(x) \neq 0\) for all \(x \in [0,1/2].\) Notice that the family \(\{p \cdot B \colon p \text{ is a polynomial}\}\) is dense in \(C[0,1/2]\) since the polynomials are dense and the multiplication operator \(M_B\) is invertible on  \(C[0,1/2].\) Now,  every function \(\varphi \in C[0,1/2]\) can be extended to a function \(\widetilde{\varphi}\in L^\infty(\mu) \) given  by the expression
\[
\widetilde{\varphi}(z) = \left \{
\begin{array}{rl} \varphi(z), & \text{if } z \in [0,1/2],\\
				0, & \text{if } z \in \{z_k \colon k \geq 1\}.
\end{array}
\right.
\]
We claim that \(M_{\widetilde{\varphi}} \in \overline{{\mathcal R}}^\sigma\) for every \(\varphi \in C[0,1/2].\) Indeed, let \((p_n)\) be a sequence of polynomials such that \(p_n \cdot B \to \varphi\) uniformly on \([0,1/2].\) Since \(p_n \cdot B\) vanishes identically on \(\{z_k \colon k \geq 1\},\) it follows that \(p_n \cdot B \to \widetilde{\varphi}\) on \(L^\infty(\mu),\)  and since \(p_n \cdot B \in H^\infty(\mathbb{D}),\) we obtain \(M_{p_n \cdot B} \in \overline{{\mathcal R}}^\sigma,\) and we conclude that \(M_{\widetilde{\varphi}} \in \overline{{\mathcal R}}^\sigma.\) Next, take a \(B_1 \in  H^\infty(\mathbb{D})\) such that \(B_1(z_1)=1\) and \(B_1(z_k) =0\) for all \(k \geq 2.\) Then,  consider the function  defined as \(\varphi=B_1-B_1 \cdot \chi_{[0,1/2]}.\) Since  \(B_1 \in \overline{{\mathcal R}}^\sigma\) and since \(B_1 \cdot \chi_{[0,1/2]} =\widetilde{B}_1|_{[0,1/2]}  \in \overline{{\mathcal R}}^\sigma,\) we obtain \(M_\varphi \in  \overline{{\mathcal R}}^\sigma.\) Notice that \(M_\varphi\) is a rank one operator, since \(\varphi(z)=1\) for \(z=z_1\) and \( \varphi(z)=0\) for \(z \neq z_1.\)
\end{proof}

We have shown so far that the algebra \(\overline{{\mathcal R}}^\sigma\) is localizing because it contains a rank one operator. This is all we need for our construction, although something stronger can be said, namely, that \(\overline{{\mathcal R}}^\sigma = L^\infty(\mu).\) Since the measure \(\mu\) contains many atoms, there are many  rank one  operators in \(\overline{{\mathcal R}}^\sigma .\)

\begin{lemma} \label{interpol} If the sequence \((z_k)\) satisfies  the Blaschke condition 
\[
\displaystyle{\sum_{k=1}^\infty (1-|z_k|) < \infty}
\]
then the algebra \({\mathcal R}\)  is dense in \(L^\infty(\mu)\) with respect to the weak operator topology.
\end{lemma}

\begin{proof} We claim that for every sequence of scalars \(\alpha=(\alpha_k)\) with \(\alpha_k =0\) for all \(k >N\) there is a \(\varphi \in \overline{{\mathcal R}}^\sigma\) such that \(\varphi(x)=0\) for all \(x \in [0,1/2]\) and such that \(\varphi(z_k)=\alpha_k\) for all \(k \geq 1.\) Indeed, take a sequence \((B_k)\) in \(H^\infty(\mathbb{D})\) such that \(B_k(z_j)=1\) if \(j=k\) and \(B_k(z_j)=0\) if \(j \neq k.\) Then, the required conditions are fulfilled by the function
\[
\varphi(z)=\sum_{k=1}^N \alpha_k B_k(z).
\]
Next, we claim that for every \(\psi \in C[0,1/2]\) there is a \(\Phi \in \overline{{\mathcal R}}^\sigma\) such that \(\Phi(x)= \psi(x)\) for all \(x \in [0,1/2]\) and \(\Phi(z_k)=\alpha_k\) for all \(k \geq 1.\) Indeed, let  \(\varphi\) be a function as above and notice that the function \(\Phi=\varphi + \widetilde{\psi}\) does the job. Finally, we show that \({\mathcal R}\) is dense in \(L^\infty(\mu)\) with respect to the weak operator topology. Take a function \(\varphi \in L^\infty(\mu).\) There is a sequence of functions \((\psi_n)\) in \(C[0,1/2]\) such that  \(\psi_n \to \varphi |_{[0,1/2]}\) in the weak-\(\ast\) topology. Then, consider for every \(n \geq 1\) the sequence of scalars \(\alpha^n=(\alpha^n_k)_{k \geq 1}\) defined by
\[
\alpha^n_k = \left \{
\begin{array}{rl} \varphi(z_k), & \text{if } 1 \leq k \leq n,\\
				0, & \text{if } k >n.
\end{array}
\right.
\]
Let \(\Phi_n \in \overline{{\mathcal R}}^\sigma \) be the function associated with \(\psi_n\) and \(\alpha^n,\) that is, \(\Phi_n= \varphi + \widetilde{\psi}_n,\) and notice that \(\Phi_n \to \varphi\) in the weak-\(\ast\) topology, as we wanted. 
\end{proof}
\begin{theorem} The algebra \({\mathcal R}\) is not localizing on the Hilbert space \(L^2(\mu).\)
\end{theorem}
\begin{proof}  Let \(f_0 \in L^2(\mu)\) and consider a closed ball \(B= \{ f \in L^2(\mu) \colon \|f-f_0\|_2 \leq \varepsilon\}.\) We proceed by contradiction. Suppose that the algebra \({\mathcal R}\) and the ball \(B\)  satisfy the condition  in the definition of a localizing algebra. 

{\bf Claim:}  \(f_0=0\) almost everywhere on \([0,1/2].\) Indeed, there is a \(\delta >0\) such that, for every Borel subset \(A \subseteq  [0,1/2],\) the condition \(\mu(A) < \delta \) implies \(\|f_0 \cdot \chi_A\|_2 < \varepsilon.\) Then, we may construct a sequence \((A_n)\) of   independent Borel subsets of \([0,1/2]\) such that \(\mu(A_n)=\delta/2\) for each \(n \geq 1\). Next, we define a sequence of functions \((f_n)\) inside the ball \(B\) by the expression \(f_n=f_0 - f_0 \cdot \chi_{A_n}.\) Now, there exists a subsequence \((f_{n_j}),\)  a sequence \((\varphi_j)\)  in \({\rm ball}({\mathcal R})\) and a nonzero  \(g \in L^2(\mu)\) such that \(\varphi_j f_{n_j} \to g\)  in \(L^2(\mu).\) Then, apply Montel's theorem to obtain another subsequence \((\varphi_{j_k})\) and a function \(\varphi \in H^\infty(\mathbb{D})\) such that \(\varphi_{j_k} \to \varphi\) uniformly on compact subsets of \(\mathbb{D}.\) Since \(\varphi_{j_k} f_{n_{j_k}} \to g\) in \(L^2(\mu),\) we may extract a further subsequence \((\varphi_{j_{k_l}})\) such that \(\varphi_{j_{k_l}} f_{n_{j_{k_l}}} \to g\) almost everywhere.  Next, apply the Borel-Cantelli Lemma to the sequence \((A_{n_{j_{k_l}}})\), and conclude that there is a measurable subset \(Z_1 \subseteq [0,1/2]\) such that \(\mu(Z_1)=0\) and such that  the set \(\{l \geq 1: x \in A_{n_{j_{k_l}}}\}\) is infinite for every \(x \in Z_1^c\). It follows  that \(g(x)=0\)  almost everywhere on \([0,1/2].\) Then, apply once again the Borel-Cantelli Lemma to the sequence \((A^c_{n_{j_{k_l}}})\), and conclude that there is a measurable subset \(Z_2 \subseteq [0,1/2]\) such that \(\mu(Z_2)=0\) and such that  the set \(\{l \geq 1: x \in A^c_{n_{j_{k_l}}}\}\) is infinite for every \(x \in Z_2^c\).  It follows that \(f_0(x)\varphi(x)=g(x)=0\) almost everywhere  on \([0,1/2].\)  Since \(\varphi\) is  an analytic function, there are two possibilities: either \(\varphi \equiv 0\) or \(f_0 \equiv 0\) almost everywhere on \([0,1/2].\) Now we prove that the first possibility leads to a contradiction. We have \(|f_{n_{j_k}} | \leq |f_0|,\) so that \(|\varphi_{j_k}  f_{n_{j_k}} | \leq | \varphi_{j_k}|  \cdot  |f_0| \to 0,\) and it follows from the bounded convergence theorem that \(\varphi_{j_k}  f_{n_{j_k}} \to 0\) in \(L^2(\mu).\) Hence, \(g=0,\) and we arrived at a contradiction. The proof of our claim is now complete. 

Now consider the sequence \((f_n)\) in \(B\) defined as \(f_n=f_0 + \chi_{A_n}.\)  Since \({\mathcal R}\) is localizing, there exists a subsequence \((f_{n_j}),\)  a sequence \((\varphi_j)\)  in \({\rm ball}({\mathcal R}),\) and a nonzero  \(g \in L^2(\mu)\) such that \(\varphi_j f_{n_j} \to g\)  in \(L^2(\mu).\)  Then, apply Montel's theorem to obtain another subsequence \((\varphi_{j_k})\) and a function \(\varphi \in H^\infty(\mathbb{D})\) such that \(\varphi_{j_k} \to \varphi\) uniformly on compact subsets of \(\mathbb{D}.\) Next, extract a further subsequence \((\varphi_{j_{k_l}})\) such that \(\varphi_{j_{k_l}} f_{n_{j_{k_l}}} \to g\) almost everywhere on \(\mathbb{D}.\)   Then, apply the Borel-Cantelli Lemma to the sequence \((A_{n_{j_{k_l}}})\), and conclude that there is a measurable subset \(Z_1 \subseteq [0,1/2]\) such that \(\mu(Z_1)=0\) and such that  the set \(\{l \geq 1: x \in A_{n_{j_{k_l}}}\}\) is infinite for every \(x \in Z_1^c\). Hence, \(\varphi_{j_{k_l}} f_{n_{j_{k_l}}} \to 0\)  almost everywhere on \([0,1/2].\)  Then, apply once again the Borel-Cantelli lemma to the sequence  \((A^c_{n_{j_{k_l}}})\), and conclude that there is a measurable subset \(Z_2 \subseteq [0,1/2]\) such that \(\mu(Z_2)=0\) and such that  the set \(\{l \geq 1: x \in A^c_{n_{j_{k_l}}}\}\) is infinite for every \(x \in Z_2^c\). Hence, \(\varphi_{j_{k_l}} f_{n_{j_{k_l}}} \to \varphi(x)\)  almost everywhere on \([0,1/2].\) Therefore, we get \(\varphi(x)=0\) almost everywhere on \([0,1/2],\) and since \(\varphi\) is analytic, we conclude that \(\varphi(z)=0\) for all \(z \in \mathbb{D}.\) Finally, \(|\varphi_{j_{k_l}} f_{n_{j_{k_l}}}| \leq (1+|f_0|) |\varphi_{j_{k_l}}| \to 0,\) so that it follows from the bounded convergence theorem that \(g=0,\) and a contradiction has arrived.
\end{proof}

\section{Bounded analytic multipliers on the Bergman space and the Hardy space}
\label{bam}
\noindent
We now consider algebras of multiplication operators on Hilbert spaces of analytic functions. We identify every bounded analytic function \(\varphi \in H^\infty(\mathbb{D})\)  with the multiplication operator \(M_\varphi\) defined either on the Bergman space \(A^2(\mathbb{D})\) or on the Hardy space \(H^2(\mathbb{D})\) by the expression \((M_\varphi f)(z)=\varphi(z)f(z).\) Then, \(H^\infty(\mathbb{D})\) becomes a subalgebra of both \(\mathcal{B}(A^2(\mathbb{D}))\)  and \(\mathcal{B}(H^2(\mathbb{D}))\) under this identification.
\begin{theorem}
\label{bergman}
The algebra \(H^\infty(\mathbb{D})\) of all bounded analytic functions on the unit disc is not localizing on the Bergman space \(A^2(\mathbb{D})\).
\end{theorem}

\begin{proof}
Let \(f_0  \in A^2(\mathbb{D})\) and consider a closed ball \(B =\{f \in A^2(\mathbb{D}): \| f-f_0\|_2  \leq \varepsilon \}.\) We must show that the algebra \(H^\infty(\mathbb{D})\) and the ball \(B\) do not satisfy the condition in the definition of a localizing algebra. First of all, consider the orthonormal basis \((e_n)\) of \(A^2(\mathbb{D})\) formed by the monomials \(e_n(z)=(n+1)^{1/2}z^n\).  Next, consider the sequence \((f_n)\) in \(B\) defined by the expression \(f_n=f_0+ \varepsilon e_n\). Now, suppose that  there is a subsequence \((f_{n_j})\), a function \(g_0  \in A^2(\mathbb{D})\), and a sequence of functions \((\varphi_j)\) in \(H^\infty(\mathbb{D})\) such that \(\|\varphi_j\|_\infty \leq 1\) and \(\|\varphi_jf_{n_j}-g_0\|_2 \rightarrow 0\) as \(j \rightarrow \infty\). Thus, it suffices to show that \(g_0=0\). Then, apply Montel's theorem to extract a further subsequence \((\varphi_{j_k})\) that converges uniformly on compact sets to some \(\varphi \in H^\infty(\mathbb{D})\). The bounded convergence theorem gives  \(\|\varphi_{j_k}f_0-\varphi f_0\|_2 \rightarrow 0\) as \(k \rightarrow \infty\).  Hence, \(\|\varepsilon \varphi_{j_k} e_{n_{j_k}}-(g_0-\varphi f_0)\|_2 \rightarrow 0\) as \(k \rightarrow \infty\). Notice that \(|\varphi_j(z)e_{n_j}(z)| \leq (n_j+1)^{1/2}|z|^{n_j} \rightarrow 0 \text{ as } j \rightarrow \infty\) for each \(z \in \mathbb{D},\)  and this  gives \(g_0=\varphi f_0\) and \(\varphi_{j_k} e_{n_{j_k}} \|_2 \to 0.\)  Since our aim is to prove that \(g_0=0,\) it is enough to show that \(\varphi=0\). Fix an integer \(m \geq 0\) and use Cauchy's integral formula for the derivatives to get, for each  \(1/2 \leq r < 1\) and  each \(j \geq 1,\)  
\[
|\varphi_j^{(m)}(0)| \leq \frac{m!}{2 \pi r^m} \int_0^{2\pi} |\varphi_j(re^{i\theta})| \,d\theta \leq \frac{m!2^m}{(2 \pi)^{1/2}}  \left ( \int_0^{2\pi} |\varphi_j(re^{i\theta})|^2 \,d\theta \right )^{1/2}.
\]
Since \(|e_{n_j}(re^{i\theta})|^2=(n_j+1)r^{2n_j}\), squaring both sides in the above inequality gives 
\[
|\varphi_j^{(m)}(0)|^2 \cdot  \frac{2(n_j+1)r^{2n_j+1}}{(m!)^2 4^m} \leq \frac{1}{\pi} \int_0^{2\pi} |\varphi_j(re^{i\theta})e_{n_j}(re^{i \theta})|^2r \,d\theta.
\]
Now, integrating this inequality over the interval \(1/2 \leq r < 1\) leads to 
\[
|\varphi_j^{(m)}(0)|^2 \cdot  \frac{2(n_j+1)}{(m!)^2 4^m} \int_{1/2}^1 r^{2n_j+1}dr \leq \frac{1}{\pi} \int_{1/2}^1 \int_0^{2\pi} |\varphi_j(re^{i\theta})e_{n_j}(re^{i \theta})|^2r \, d\theta dr,
\]
and from here we obtain the estimate
\[
|\varphi_j^{(m)}(0)|^2 \cdot  \frac{1-1/4^{n_j+1}}{(m!)^2 4^m} \leq \frac{1}{\pi} \int_0^1 \int_0^{2\pi} |\varphi_j(re^{i\theta})e_{n_j}(re^{i \theta})|^2r \, d\theta dr=\|\varphi_j e_{n_j}\|_2^2.
\]
Finally, passing a subsequence \((\varphi_{j_k})\) and taking limits as \(k \rightarrow \infty\) yields
\[
\frac{|\varphi^{(m)}(0)|}{m!2^m} = \lim_{k \rightarrow \infty} |\varphi_{j_k}^{(m)}(0)| \cdot \frac{(1-1/4^{n_{j_k}+1})^{1/2}}{m!2^m}  \leq \lim_{k \rightarrow \infty} \|\varphi_{j_k} e_{n_{j_k}}\|_2=0,
\]
so that \(\varphi^{(m)}(0)=0\) for each \(m \geq 0\), that is, \(\varphi=0,\) as we wanted.
 \end{proof}

\begin{theorem}
\label{hardy}
The algebra \(H^\infty(\mathbb{D})\) of all bounded analytic functions on the unit disc is not localizing on the Hardy space \(H^2(\mathbb{D})\).
\end{theorem}

\noindent
Before we proceed with the proof of  Theorem \ref{hardy}, we  state and prove  several lemmas. We shall denote by \(\mu\) the normalized Haar measure on the torus \(\mathbb{T}=\{z \in \mathbb{C}: |z|=1\}\). Also, we shall denote by \((e_n)\) the orthonormal basis in \(L^2(\mathbb{T})\) of the functions defined by the expression \(e_n(z)=z^n\) for every \(n \in \mathbb{Z}.\) Finally, for every measurable set \(B \subseteq \mathbb{T}\), we shall consider the preimages \(e_n^{-1}(B)=\{z \in \mathbb{T}:z^n \in B\}.\)

\begin{lemma}
\label{uniform}
If \(A,B \subseteq \mathbb{T}\) is any pair of measurable sets then we have
\[
\lim_{n \rightarrow \infty} \mu(A \cap e_n^{-1}(B))=\mu(A) \mu(B).
\]
\end{lemma}

\begin{proof}
First of all,  it is plain that \( \chi_{e_n^{-1}(B)}(z)=\chi_B(z^n)\) for every \(z \in \mathbb{T}\), and from this fact it follows that the Fourier coefficients for the characteristic function of the preimage \(e_n^{-1}(B)\) are given by the expression
\[
\widehat{\chi}_{e_n^{-1}(B)}(m)= \left \{ \begin{array}{ll} \widehat{\chi}_B(m/n), & \text{ if  \(m\) is a multiple of \(n\),}\\
											0,				& \text{ otherwise.} \end{array} \right.							
\]
Next, use Parseval's identity to obtain
\begin{eqnarray*}
\mu(A \cap e_n^{-1}(B)) & = & \int_\mathbb{T} \chi_A(z) \cdot \chi_{e_n^{-1}(B)}(z) \, d\mu (z)\\
& = & \sum_{m 	\in \mathbb{Z}} \widehat{\chi}_A(m) \cdot \widehat{\chi}_{e_n^{-1}(B)}(m) =  \sum_{k	\in \mathbb{Z}} \widehat{\chi}_A(nk) \cdot \widehat{\chi}_{B}(k)\\
& = & \widehat{\chi}_A(0) \cdot  \widehat{\chi}_B(0) +  \sum_{k	\in \mathbb{Z} \backslash \{0\}} \widehat{\chi}_A(nk) \cdot \widehat{\chi}_{B}(k)\\
& = & \mu(A)\mu(B) + \sum_{k	\in \mathbb{Z} \backslash \{0\}} \widehat{\chi}_A(nk) \cdot \widehat{\chi}_{B}(k).
\end{eqnarray*}
Finally, use  the Cauchy-Schwartz inequality to conclude that
\begin{eqnarray*}
|\mu(A \cap e_n^{-1}(B))-\mu(A)\mu(B)| 
& \leq &
\| \chi_B\|_2 \cdot \left ( \sum_{k \in \mathbb{Z} \backslash \{0\}} | \widehat{\chi}_A(nk) |^2 \right )^{1/2}\\
& \leq & \mu(B)^{1/2} \cdot \left ( \sum_{|k| \geq |n|} | \widehat{\chi}_A(k) |^2 \right )^{1/2},
\end{eqnarray*}
and notice that the last expression approaches zero as \(n \rightarrow \infty\), as we wanted.
\end{proof}

\noindent
The following result has the same flavour as the Borel-Cantelli Lemma, which cannot be applied here because the measurable sets under consideration are not necessarily independent.
\begin{lemma}
\label{surely}
Let \(A \subseteq \mathbb{T}\) be a measurable set with \(\mu(A)>0\), let \((n_j)\) be an increasing sequence of positive integers, and let
\(A_j=\{ z \in \mathbb{T}: z^{n_j} \in A\}\). Then we have
\[
\mu \left ( \bigcap_{k=1}^\infty \bigcup_{j=k}^\infty A_j \right ) =1.
\]
\end{lemma}

\begin{proof}
Taking complements, the above statement is equivalent to saying that
\[
\mu \left ( \bigcup_{k=1}^\infty \bigcap_{j=k}^\infty (\mathbb{T} \backslash A_j ) \right ) =0.
\]
Hence, it suffices to show for every \(k \geq 1\) that
\[
\mu \left ( \bigcap_{j=k}^\infty (\mathbb{T} \backslash A_j ) \right ) =0.
\]
Fix \(k_0 \geq 1\),  and for each \(k \geq k_0\),  consider the quantity
\[
\alpha_k = \mu  \left ( \bigcap_{j=k_0}^k (\mathbb{T} \backslash A_j ) \right ).
\]
Then, \((\alpha_k)\) is a decreasing sequence of nonnegative numbers. Set  \(\alpha = \lim \alpha_k\). We must show that \(\alpha=0\). Fix \(k \geq k_0\) and notice that, for each \(l \geq k\),
\[
\bigcap_{j=k_0}^l (\mathbb{T} \backslash A_j )  \subseteq  \left (  \bigcap_{j=k_0}^k (\mathbb{T} \backslash A_j )\right ) \cap (\mathbb{T} \backslash A_l).
\]
We have \(\mathbb{T} \backslash A_l = e_{n_l}^{-1}(\mathbb{T} \backslash A)\), so that Lemma \ref{uniform} can be applied to obtain
\begin{eqnarray*}
\alpha & = & \lim_{l \rightarrow \infty} \mu \left ( \bigcap_{j=k_0}^l (\mathbb{T} \backslash A_j ) \right ) 
\leq  \lim_{l \rightarrow \infty} \mu  \left [ \left (  \bigcap_{j=k_0}^k (\mathbb{T} \backslash A_j )\right ) \cap (\mathbb{T} \backslash A_l) \right ]\\
& = & \mu \left (  \bigcap_{j=k_0}^k (\mathbb{T} \backslash A_j )\right ) \cdot \mu (\mathbb{T} \backslash A)= \alpha_k \cdot \mu (\mathbb{T} \backslash A).
\end{eqnarray*}
Finally, taking limits as \(k \rightarrow \infty\) leads to the inequality  \(\alpha \leq \alpha \cdot  \mu (\mathbb{T} \backslash A)\), and since \(\mu(\mathbb{T} \backslash A) < 1\), we conclude that \(\alpha=0\), as we wanted.
\end{proof}

\begin{lemma}
\label{holomorphic}
If \(0 < \delta < 1\) then there is an open set \(U \supseteq \overline{\mathbb{D}}\) and there is a holomorphic function \(h \in H(U)\) such that
\begin{enumerate}
\item \(h(1)=0\),
\item \(|h(z)| \leq 1\) for each \(z \in \overline{\mathbb{D}}\),
\item \(\mu (\{z \in \mathbb{T}: |h(z)-1| > \delta \}) < \delta\).
\end{enumerate}
\end{lemma} 

\begin{proof}
Let \(0 < r < 1\) to be chosen later on, and consider the function \(h\) defined by the expression
\[
h(z)=\frac{r(1-z)}{1-rz}.
\]
It is plain that \(h\) is holomorphic on  \(\mathbb{C} \backslash \{1/r\} \supseteq \overline{\mathbb{D}}\) and that \(h(1)=0\). Since \(h\) is a Moebius transformation, it is easy to check that 
\(h(\overline{\mathbb{D}})\) is a disc of radius \(r/(1+r)\) centered at \(r/(1+r)\). It follows that \(|h(z)| \leq 1 \) for each \( z \in \overline{\mathbb{D}}\). Now, consider the arc \(A=\{e^{i\theta}: |\theta| < \pi \delta/2\}\) and notice that \(\mu(A)=\delta/2\). Thus, it suffices to show that \(\{z \in \mathbb{T}: |h(z)-1| > \delta\} \subseteq A\) for a suitable choice of  \(r.\) Consider the compact set \(K=\{rz: r \in [0,1], \, z \in \mathbb{T}\backslash A\}\). Since \(1 \notin K\), there is an \(\eta > 0\)  such that \(|1-rz| \geq \eta\) for every \(r \in [0,1]\) and for every \(z \in \mathbb{T} \backslash A\). Thus, for each \(z \in \mathbb{T} \backslash A\) we have
\[
|h(z)-1| =\frac{1-r}{|1-rz |} \leq \frac{1-r}{\eta} < \delta,
\]
as long as \(r\) is chosen to be close enough to 1.
\end{proof}

\begin{proof}[Proof of Theorem \ref{hardy}]
Let  \(f_0 \in H^2(\mathbb{D})\) and consider the closed ball  \(B=\{f \in H^2(\mathbb{D}): \|f-f_0\| \leq \varepsilon\}\). We must show that the algebra \(H^\infty(\mathbb{D})\) and the ball \(B\) do not satisfy the conditions in the definition. Let \(\delta>0\) to be chosen later on,  let \(h\) be a holomorphic function as in Lemma \ref{holomorphic}, and define a sequence of functions \((f_n)\) in  \(H^2(\mathbb{D})\) by the expression
\[
f_n(z)=f_0(z)h(z^n), \qquad z \in \mathbb{D}, \quad n \geq 1.
\]
We claim that \(f_n \in B\) for each \(n \geq 1\), provided that \(\delta>0\) is suitably chosen. Indeed, consider the  measurable sets 
\[
A=\{z \in \mathbb{T}: |h(z)-1|>\delta\} \quad \text{and} \quad A_n=\{z \in \mathbb{T}: |h(z^n)-1|>\delta\},
\]
and notice that \(A_n=e_n^{-1}(A)\), so that \(\mu(A_n)=\mu(A)<\delta\). Then we get
\begin{eqnarray*}
\|f_n-f_0\|_2^2 & = & \int_\mathbb{T} |h(z^n)-1|^2 \cdot |f_0(z)|^2 \,d\mu(z)\\
& \leq & \int_{\mathbb{T} \backslash A_n} \delta^2 |f_0(z)|^2d\mu(z) + \int_{A_n} 4 |f_0(z)|^2\, d\mu(z)\\
& \leq & \delta^2 \|f_0\|_2^2 + \int_{A_n} 4 |f_0(z)|^2 \,d\mu(z).
\end{eqnarray*}
Now, choose \(\delta > 0\) such that \(\delta^2 \|f_0\|_2^2 < \varepsilon^2/2\), and with the property that, for each measurable set \(B \subseteq \mathbb{T}\), the condition \(\mu(B) < \delta\)  implies that
\[
\int_B4|f_0(z)|^2 \,d\mu(z) < \varepsilon^2/2.
\]
Hence, \(\|f_n-f_0\|_2^2 < \varepsilon^2\), and the proof of our claim is over. Next, suppose that there is a subsequence \((f_{n_j})\), a sequence \((\varphi_j)\) in \(H^\infty(\mathbb{D})\), and a function \(g \in H^2(\mathbb{D})\) such that \( \|g-\varphi_jf_{n_j}\|_2 \rightarrow 0\) as \(j \rightarrow \infty\). Then, it suffices to show that \(g(z)=0\) for almost every \(z \in  \mathbb{T}\). We may assume, extracting a subsequence if necessary, that \(\varphi_j(z)f_{n_j}(z) \rightarrow g(z)\) as \(j \rightarrow \infty\) for almost every \(z \in \mathbb{T}\). Thus, there is a measurable set \(N_0\ \subseteq \mathbb{T}\) such that \(\mu(N_0)=0\) and such that, for every \(z \in \mathbb{T} \backslash N_0\), we have
\begin{eqnarray*}
|g(z)| & = & \liminf_{j \rightarrow \infty} |\varphi_j(z)| \cdot |f_{n_j}(z)| \\
& \leq & \liminf_{j \rightarrow \infty}  |f_{n_j}(z)| \\
& = & |f_0(z)| \cdot \liminf_{j \rightarrow \infty} |h(z^{n_j})|.
\end{eqnarray*}
Since \(h\) is continuous at \(z=1\), for every integer \(m \geq 1\) there is an open set \(G_m \subseteq \mathbb{T}\) such that \(1 \in G_m\) and \(|h(z)|<1/m\) for each \(z \in G_m\). Now, apply Lemma \ref{surely} to get a measurable set \(N_m \subseteq \mathbb{T}\) with \(\mu(N_m)=0\) and  such that \(z^{n_j} \in G_m\) infinitely often for each \(z \in \mathbb{T} \backslash N_m\). Therefore, \(\liminf |h(z^{n_j}| \leq 1/m\) as \(j \rightarrow \infty\) for  each \(z \in \mathbb{T} \backslash N_m\). Finally, consider the countable union of measurable sets
\[
N=\bigcup_{m=0}^\infty N_m,
\]
and notice that \(\mu(N)=0\). If \(z \in \mathbb{T} \backslash N\) then \(|g(z)| \leq  |f_0(z)|/m\) for every integer \(m \geq 1.\) We  conclude from this inequality that \(g(z)=0\) for every \(z \in \mathbb{T} \backslash N,\) as we wanted.
\end{proof}
\section{Algebras generated by diagonal operators}
\label{diag}
\noindent
Now we turn our attention to the algebra generated by a single normal operator \(T.\) The spectral theorem ensures that  there is measure \(\mu\) of compact support on the Borel subsets of the complex plane such that  \(T\) is unitarily equivalent to a multiplication by a bounded measurable function on \(L^2(\mu).\) Then the algebra generated by \(T\) may be regarded as a subalgebra of \(L^\infty(\mu),\) and in view of Theorem \ref{atomic}, if such an algebra is localizing then the measure \(\mu\) must have an atom. Now we focus on the  extreme case that \(\mu\) is a purely atomic measure, so that \(T\) is a diagonal operator.

Let \((z_j)\) be a sequence of complex numbers in the closed  unit disc. Consider the diagonal operator \(T={\rm diag}(z_j),\)  that is, \(Te_j=z_je_j,\) where \((e_j)\) is an orthonormal basis of an infinite dimensional, separable complex Hilbert space \(H\) and \(j\) runs through the non negative integers. Suppose that  \(z_j \neq z_k\) whenever \(j \neq k.\) Then, let \({\mathcal R}=\{p(T): p \text{ is a polynomial}\}\) denote the unital algebra generated by \(T.\)  In this section, some conditions  are given for the algebra \({\mathcal R}\) to be localizing.
\begin{proposition}
\label{ism}
If \(\,|z_{j_0}|=1\) for some \(j_0 \geq 0\) then \(\,\overline{{\rm ball}({\mathcal R})}^\sigma\) contains a rank one operator.
\end{proposition}
\begin{proof}
Consider the sequence of polynomials \((p_n)\) defined by the expression 
\[
p_n(z):=\left ( \frac{\overline{z}_{j_0} z+1}{2} \right )^n.
\]
Then \(\|p_n\|_\infty \leq 1,\) so that \(p_n(T) \in {\rm ball}({\mathcal R}).\) Moreover, \((p_n)\) converges pointwise to  the function \(f\) defined by \(f(z)=0\) if \(z \neq z_{j_0}\) and \(f(z_{j_0})=1.\) Therefore, the sequence of operators \((p_n(T))\) converges in the weak operator topology to the rank one operator \(e_{j_0} \otimes e_{j_0}.\)
\end{proof}
\noindent
Recall that the spectrum of \(T\) is the compact set 
\[
\sigma(T)=\overline{\{z_j: j \geq 0\}}.
\]
\begin{proposition}
If \(\sigma(T)\) has empty interior and \({\mathbb C} \backslash \sigma (T)\) is connected, then \( \overline{{\rm ball} ({\mathcal R})}^\sigma \) is the set of  all diagonal operators of the form \({\rm diag}(\lambda_j),\) for some \((\lambda_j) \in \ell_\infty\) with  \(\|(\lambda_j)\|_\infty \leq 1.\) In particular, \( \overline{{\rm ball} ({\mathcal R})}^\sigma \) contains a nonzero compact operator.
\end{proposition}
\begin{proof} We prove the non trivial inclusion. Let \((\lambda_j) \in \ell_\infty\) with \(\|(\lambda_j)\|_\infty \leq 1,\) and  for every \(n \geq 1,\) choose a continuous function \(f_n : \sigma(T) \to \mathbb{C}\) with \(f_n(z_j)=\lambda_j\) whenever   \(1 \leq j \leq n\)  and   \(|f(z)| \leq 1\) for all \(z \in \sigma(T).\) It follows from Mergelyan's theorem that  for every \(n \geq 1\) there is a polynomial \(p_n(z)\) such that \(|p_n(z)| \leq 1\) for each \(z \in \sigma(T)\) and such that \(|p_n(z_j)-\lambda_j| < 1/n\) whenever \(1 \leq j \leq n.\) Finally, the sequence of diagonal operators \((p_n(T))\) lies inside \({\rm ball}({\mathcal R})\) and it converges to the diagonal operator \({\rm diag}(\lambda_j)\) in the strong operator topology.
\end{proof}
\noindent
The rest of this section deals with diagonal operators \(T\) with the property that \(\sigma(T) \supseteq \partial {\mathbb D}.\) We make this  assumption because it allows us to control the norm of an operator in the algebra generated by \(T.\) Indeed, if \(p\) is a polynomial then it follows from the maximum modulus principle that 
\[
\|p(T)\|= \sup\{|p(z)| \colon z \in \sigma(T)\}= \sup \{|p(z)| \colon z \in \mathbb{D}\}= \|p\|_\infty.
\]
Notice that Proposition \ref{ism} allows us to discard the case  \(\,|z_{j_0}|=1\) for some \(j_0 \geq 0,\) so that from now on we shall assume  \(|z_j| < 1\) for all \(j \geq 0.\)
\begin{proposition} The closure of the unit ball of \({\mathcal R}\)  in the weak operator topology  is the set of   all diagonal operators of the form  \(\,{\rm diag}(f(z_j)),\) where \(f \in H^\infty({\mathbb D})\) and  \(\|f\|_\infty \leq 1.\)
\end{proposition}
\begin{proof}
First, let \(R \in \overline{{\rm ball}({\mathcal R})}^\sigma.\)  Since \(H\) is separable, the weak operator topology is metrizable on bounded subsets of \({\mathcal B}(H),\) and therefore, there exists a sequence of polynomials \((p_n)\) such that \(\|p_n(T)\| \leq 1\) and  \(p_n(T) \rightarrow R\) in the weak operator topology. Now,  \(p_n(T)\) is a diagonal operator with diagonal sequence \((p_n(z_j))\) so that \(\|p_n\|_\infty = \|p_n(T)\| \leq 1.\) Then, it follows from Montel's theorem that there is a subsequence \((p_{n_k})\) that converges uniformly on compact subsets of \({\mathbb D}\) to some function \(f \in H^\infty({\mathbb D})\) with \(\|f\|_\infty \leq 1.\) Therefore, 
\[
\langle Re_j,e_l \rangle = \lim_{k \rightarrow \infty} \langle p_{n_k}(T)e_j,e_l \rangle = \lim_{k \rightarrow \infty} \langle p_{n_k}(z_j)  e_j, e_l \rangle =\langle f(z_j) e_j, e_l \rangle.
\] 
Thus, \(R={\rm diag}(f(z_j)),\) as we wanted.  Next, let \(f \in H^\infty({\mathbb D})\) with \(\|f\|_\infty \leq 1,\) and let  \(R={\rm diag}(f(z_j)).\) Then,  there is a sequence of polynomials \((p_n)\) such that \(\|p_n\|_\infty \leq 1\) and \(p_n \to f\) uniformly on compact subsets of \({\mathbb D}.\) We can take for instance the sequence of polynomials \(p_n=F_n \ast f,\) where \((F_n)\) is the sequence of the  Fej\'er kernels, that is,
\[
p_n(z)= \sum_{j=0}^n \left (1-\frac{j}{n+1} \right ) \hat{f} (j)z^j.
\]
Thus,  \(\|p_n(T) \|\leq 1\) and for every \(j,k \geq 0\) we have 
\[
\lim_{n \to \infty} \langle p_n(T)e_j,e_k \rangle = \lim_{n \rightarrow \infty}\langle p_{n}(z_j)  e_j, e_k \rangle =\langle f(z_j) e_j, e_k \rangle=\langle Re_j,e_k \rangle.
\]
This shows that \(p_n(T) \to R\) in the weak operator topology, so that \(R \in \overline{{\rm ball}({\mathcal R})}^\sigma, \) as we wanted. 
\end{proof}
\begin{corollary} The  following conditions are equivalent:
\begin{enumerate}
\item \(\overline{{\rm ball}({\mathcal R})}^\sigma\) contains a non zero compact operator,
\item  there exists \(f \in H^\infty({\mathbb D})\) such that \(\|f\|_\infty \leq1,\) with \(f(z_j) \neq 0\) for some \(j \geq 0\) and \(\displaystyle{\lim_{j \rightarrow \infty}f(z_j)=0.}\)
%\item  \(\forall \varepsilon >0\) there exists \(f \in H^\infty({\mathbb D})\) such that \(\|f\|_\infty \leq1,\) with \(|f(0)| \geq 1/2,\) and  \(\displaystyle{\limsup_{j \rightarrow \infty}|f(z_j)| < \varepsilon.}\)
\end{enumerate}
\end{corollary}
\noindent
Consider the set \(\sigma(T)^\prime\) of all cluster points of the spectrum of \(T.\) The meaning of the following result is that when the part of \(\sigma(T)^\prime\) in the open unit disc is large enough, the algebra \({\mathcal R}\) fails to be localizing.

\begin{proposition} Suppose that   there is a sequence \((w_p)\) of distinct points in \(\sigma(T)^\prime \cap {\mathbb D}\) such that 
\[
\sum_{p=1}^\infty (1-|w_p|)=\infty.
\]
Then \({\mathcal R}\) fails to be a localizing algebra.
\end{proposition}
\begin{proof}
We proceed by contradiction. Suppose  \({\mathcal R}\) is  a localizing algebra and let \(B=\{x \in H: \|x-x_0\| \leq \varepsilon\}\) be a ball as in the definition. We may assume without loss of generality that \(x_0 \in H\) has finite support, say \({\rm supp}(x_0) \subseteq [0,M].\)  Then, for every \(p \geq 1\) there is a subsequence \((z_{j_{p,q}})\) such that \(\lim_{q \to \infty} z_{j_{p,q}}=w_p\) for all \(p \geq 1.\) Moreover,  the indices \(j_{p,q}\) can be chosen in such a way that \(j_{p,q} > M\) for all \(p,q \geq 1\) and \(j_{p,q} \neq j_{s,t}\) if \((p,q) \neq (s,t).\) Now, let \((\alpha_p)\) be a sequence of positive real numbers such that \(\sum_{p=1}^\infty \alpha_p^2 < \varepsilon^2,\) and consider the sequence of vectors \((x_q)\) in the ball \(B\)  defined by \(x_q:=x_0+y_q,\) where
\[
y_q:= \sum_{p=1}^\infty \alpha_p e_{j_{p,q}}.
\]
Notice that \(x_0 \perp y_q,\) because \({\rm supp}(x_0) \subseteq [0,M]\) and \(j_{p,q} >M.\) Since \({\mathcal R}\) is a localizing algebra, there is a sequence of polynomials \((f_k)\) such that \(\|f_k(T)\| \leq 1,\) and there is a subsequence \((x_{q_k})\) such that \((f_k(T)x_{q_k})\) converges in norm to some vector \(y \neq 0.\)  Since \(\|f_k\|_\infty \leq 1,\) using Montel's theorem we may assume by extracting a subsequence if necessary that \((f_k)\) converges uniformly on  compact sets  to some function \(f \in H^\infty({\mathbb D}).\) Consider the diagonal operator \(f(T):={\rm diag}(f(z_j)).\) Since the vector  \(x_0\) has finite support, the sequence \((f_k(T)x_0)\) converges in norm to \(f(T)x_0.\) Thus, the sequence \((f_k(T)y_{q_k})\)   converges in norm. Notice that \(y_q \to 0\) weakly.  Hence, \(f_k(T)y_{q_k} \to 0\) weakly, and we may conclude that \(\|f_k(T)y_{q_k})\| \to 0.\) Therefore, we have \(y=f(T)x_0.\) Finally, it follows from the bounded convergence theorem that 
\[
\sum_{p=1}^\infty \alpha_p^2|f(w_p)|^2  =  \lim_{k \to \infty} \sum_{p=1}^\infty \alpha_p^2 |f_k(z_{j_{p,q_k}})|^2 = \lim_{k \to \infty} \|f_k(T)y_{q_k}\|^2=0,
\]
and from this  identity we get  \(f(w_p)=0\) for all \(p \geq 1.\) Finally, the condition \(\sum_{p=1}^\infty (1-|w_p|)=\infty\) implies \(f (z)=0\) for all \(z \in \mathbb{D}.\) Hence, \(y=0,\) and the contradiction has arrived.
\end{proof}
We finish this section with a statement of Problem \ref{genprob} for the special case of the algebra \({\mathcal R}\) generated by a single diagonal operator. 
\begin{problem}
{\em Let \({\mathcal R}\) be the algebra generated by a single diagonal operator on an infinite dimensional, separable complex Hilbert space. Suppose that \({\mathcal R}\) is localizing. Does \(\overline{{\rm ball}({\mathcal R})}^\sigma\) contain a rank one operator, or at least, a nonzero compact operator?}
\end{problem}
\section{Extended eigenvalues and invariant subspaces}
\label{eigen}
\noindent
The first author \cite{lacruz} obtained a simple proof of Theorem \ref{quasi} that is reminiscent of Hilden's proof of a special case of the Lomonosov original result \cite{lomonosov} and  that can be adapted  to prove Theorem \ref{main}.

\begin{proof}[Proof of Theorem \ref{main}] We proceed by contradiction. Assume the commutant \(\{T\}^\prime\) is a transitive algebra. Since \(\ker T\) is invariant under  \(\{T\}^\prime\) and  since \(T \neq 0,\) we must have \(\ker T=\{0\},\) so that \(T\) is injective. Then, let \(B \subseteq E\) be a closed ball  that makes  a localizing subspace out of \({\mathcal X}.\)  We claim that there is  some constant  \(c>0\) such that for every \(x \in B\) there is an  \(X \in {\mathcal X}\) such that \(\|X\| \leq c\) and \(TXx \in B.\) Otherwise, for every \(n \in {\mathbb N}\) there is an \(x_n \in B\) such that the condition \(X \in {\mathcal X}\) and \(TXx_n \in B\) implies \(\|X\| >n.\) Since \({\mathcal X}\) is localizing, there is a subsequence \((x_{n_j})\) and there is a sequence \((X_j)\) in \({\mathcal X}\) with \(\|X_j\| \leq 1,\) and such that  \((X_jx_{n_j})\) converges in norm to some nonzero vector \(x \in E.\) Therefore,  \((TX_jx_{n_j})\) converges in norm to \(Tx.\)  Since \(T\) is injective, we have \(Tx \neq 0.\) Since \(\{T\}^\prime\) is transitive, there is an \(R \in \{T\}^\prime\)  such that \(RTx \in {\rm int}\,B.\) Hence, there is some \(j_0 \geq 1\) such that \(RTX_j x_{n_j} \in B\) for all \(j \geq j_0.\) Since \(RT=TR,\) we have \(TRX_j x_{n_j} \in B\) for all \(j \geq j_0.\) Since  \( RX_j \in  {\mathcal X},\)   the choice of the sequence \((x_n)\) implies  \(\| RX_j\| > n_j\)  for all \(j \geq j_0.\) Finally, this leads to a contradiction, because \(\|  RX_j\| \leq  \|R\|\) for all \(j \geq 1.\) This completes the proof of our claim.

Start with a vector \(x_0 \in B\) and choose an operator \(X_1 \in {\mathcal X}\) such that \(\|X_1 \| \leq c\) and \( TX_1x_0 \in B.\) Now choose another operator  \(X_2 \in {\mathcal X}\) such that \(\|X_2\| \leq c\) and \(TX_2TX_1x_0 \in B.\) Continue this ping pong game to obtain a sequence of vectors \(x_n \in  B\) and a sequence of operators \((X_n)\) in \({\mathcal X}\) such that \(\|X_n\| \leq c\) and such that
\[
x_n = TX_n \cdots TX_1x_0 = \lambda^{n(n+1)/2}\, X_n \cdots X_1T ^nx_0.
\]
Then, let \(d=\min \{\|x\| \colon x \in B\}.\) It is plain that \(d >0\) because \(0 \notin B.\)  Assume \(|\lambda| \leq 1.\) We get 
\[
d \leq \|x_n\| \leq c^n |\lambda|^{n(n+1)/2} \cdot  \|T^n \| \cdot \| x_0\|,
\]
and this gives information on the spectral radius of \(T,\) namely,
\[
r(T)= \lim_{n \to \infty} \|T^n\|^{1/n} \geq \lim_{n \to \infty}  \frac{1}{c |\lambda|^{(n+1)/2}}.
\]
If \(|\lambda| < 1\) then  we get \(r(T)=\infty,\) and if \(|\lambda|=1\) and \(T\) is quasinilpotent then we get \(r(T) \geq 1/c.\) In both cases we obtain a contradiction.
Finally, assume \(|\lambda|>1.\) Notice that
\[
x_n = TX_n \cdots TX_1x_0 = \lambda^{-n(n-1)/2} \, T^n X_n \cdots X_1  x_0.
\]
From this identity we get
\[
d \leq \|x_n\| \leq c^n  |\lambda|^{-n(n-1)/2} \cdot  \|T^n \| \cdot \| x_0\|,
\]
and once again this gives information on the spectral radius of \(T,\) namely,
\[
r(T)= \lim_{n \to \infty} \|T^n\|^{1/n} \geq \lim_{n \to \infty}  \frac{|\lambda|^{(n-1)/2}}{c }=\infty.
\]
A contradiction has arrived.
\end{proof}
The following is an example of an operator \(T\) on a Banach space \(E\) such that the family of all extended eigenoperators associated with some extended eigenvalue of \(T\) is a localizing subspace of \({\mathcal B}(E)\) although it does not contain any nonzero compact operators.

\vskip1em
\noindent
\begin{example}
{\em Let \(E=C[0,1]\) be the Banach space of  continuous functions on the unit interval endowed with the supremum norm. Then, consider the position operator \(M_t \in {\mathcal B}(E)\) defined as \((M_tf)(t)=tf(t).\)  We shall show that the set of extended eigenvalues of the position operator \(M_t\) is  the interval \((0,\infty),\)  and moreover, the extended eigenoperators associated with such extended eigenvalues belong to the class of weighted composition operators.}  
\end{example}
\begin{lemma}
\label{polynomial}
Let \(\lambda \in {\mathbb C}\) be an extended eigenvalue of  \(M_t\)  and let \(X \in {\mathcal B}(E)\) be an extended eigenoperator associated with \(\lambda.\) Then \(\lambda \neq 0\) and there is a nonzero function \(\varphi \in C[0,1]\) so that for every polynomial \(p,\)
 \[
 (Xp)(t)=\varphi(t)p(t/\lambda).
 \]
\end{lemma}
\begin{proof}
We have \(M_tX= \lambda XM_t.\) Notice that \(\lambda \neq 0\) because otherwise \(M_tX=0,\) and since \(M_t\) is injective, it follows that \(X=0.\)  Then we have \(XM_t^n =\lambda^{-n} M_t^n X,\) and setting \(\varphi =X1\) we get  
\(Xt^n = \varphi(t) (t/\lambda)^n.\) Hence, the desired identity follows by linearity. Notice that the function \(\varphi\) does not vanish identically, for otherwise \(Xp=0\) for every polynomial \(p,\) and it follows from the Weierstrass theorem that \(X=0.\)
\end{proof}
\begin{lemma}
\label{interval1}
If \(\lambda \in {\mathbb C}\) is an extended eigenvalue of \(M_t\) then \(\lambda \in (0,\infty).\)
\end{lemma}
\begin{proof}
We know  that \(\lambda \neq 0\) and there is a nonzero function \(\varphi \in C[0,1]\) such that for every polynomial \(p,\) \((Xp)(t)=\varphi(t)p(t/\lambda).\) Let \(t_0 \in [0,1]\) such that \(\varphi(t_0) \neq 0.\) Since \(\varphi\) is continuous, we may assume without loss of generality that \(t_0 \neq 0.\) We proceed by contradiction. If \(\lambda \notin [0,\infty)\) then \(t_0/ \lambda \notin [0,1]\) and  it follows from Runge's theorem that for every \(n \geq 1\) there is a polynomial \(p_n\) such that \(|p_n(t)| \leq 1\) for every \(t \in [0,1]\) and such that \(|p_n(t_0/\lambda)| \geq n.\) Hence, \(\|p_n\|_\infty \leq 1\) but \(\|Xp_n\|_\infty \geq  |\varphi(t_0)| n,\) and this is a contradiction, because \(X\) is a bounded operator.
\end{proof}
\begin{lemma}
\label{interval2}
If  \(\lambda \in (0,\infty)\) then \(\lambda\) is an extended eigenvalue of \(M_t,\) and moreover,    if   \(\varphi \in C[0,1]\) is some nonzero function  such that  \(\varphi(t)=0\) for all \(t \in [0,1] \cap (\lambda,\infty),\)  then the operator  \(X\) defined by
\[
(Xf)(t)= \left \{ \begin{array}{rl} \varphi(t)f(t/\lambda), & \text{if  } t \in [0,1] \cap [0,\lambda],\\
								0, & \text{if }t \in [0,1] \cap (\lambda,\infty),
		\end{array} \right. \qquad (\ast)
\]
is an extended eigenoperator of \(M_t\) associated with \(\lambda.\) Conversely, if \(X\) is an extended eigenoperator of \(M_t\) associated with \(\lambda\) then there is some nonzero function \(\varphi \in C[0,1]\)  such that  \(\varphi(t)=0\) for every \(t \in [0,1] \cap (\lambda,\infty)\) and such that    \(X\) is given by  the above expression.
\end{lemma}
\begin{proof} Let us suppose that an operator \(X\) is given by the  expression  \((\ast).\) Notice that \(X\) is well defined since \(t/\lambda \in [0,1]\)  for all \(t \in [0,1] \cap [0,\lambda],\) and \(Xf\) is  continuous  since \(\varphi(\lambda)=0\)  in the case \(\lambda \in (0,1).\) Also, it is clear that \(X\) is linear and bounded, with \(\|X\| \leq \|\varphi\|_\infty.\)  Then, for every \(f \in C[0,1],\) we have
\[
 (M_tXf)(t)   =  t \varphi(t)f (t/\lambda)= \lambda \varphi(t) (t/\lambda) f(t/\lambda)= \lambda (XM_t f)(t)
 \]
 if \(t \in [0,1] \cap [0,\lambda]\) and \((M_tXf)(t) =0=(XM_t f)(t)\) if \(t \in [0,1] \cap (\lambda,\infty),\) so that  \(\lambda\) is an extended eigenvalue of \(M_t\) and \(X\) is an extended eigenoperator  associated with \(\lambda.\) Conversely, if \(X\) is an extended eigenoperator of \(M_t\) associated with \(\lambda\) then it follows from Lemma \ref{polynomial} that  there is some nonzero function \(\varphi \in C[0,1]\) such  that \((Xp)(t)=\varphi(t)p(t/\lambda)\)  for every polynomial \(p.\) We  need to show that, when \(\lambda \in (0,1),\) we have  \(\varphi(t)=0\) for every \(t \in (\lambda,1].\) Indeed, if \(\varphi(t_0) \neq 0\) for some \(t_0 \in (\lambda,1]\) then we have \(Xt^n=  \varphi(t)(t/\lambda)^n, \) so that \(\|Xt^n\|_\infty \geq |\varphi (t_0)|(t_0/\lambda)^n,\)  and this is a contradiction, because \(X\) is a bounded operator. Finally, it follows from the Weierstrass approximation theorem that   \(Xf\) is given by the expression \((\ast)\) for every  \(f \in C[0,1]\) since this relationship is fulfilled   whenever \(f\) is a polynomial. 
\end{proof}

\begin{theorem} The family \({\mathcal X}\) of all extended eigenoperators of \(M_t\) associated  with an extended eigenvalue \(\lambda \in (0,\infty)\) is a localizing subspace of \({\mathcal B}(C[0,1])\) and it  does not contain any nonzero compact operators.
\label{example}
\end{theorem}
\begin{proof} First, we show that \({\mathcal X}\) is localizing. Consider the closed ball \(B=\{f \in C[0,1] \colon \|f-1\|_\infty \leq 1/2\}.\) Take a sequence \((f_n)\) in \(B.\) Notice that \(1/2 \leq |f_n(t)| \leq 3/2\) for every \(t \in [0,1].\)  Suppose that \(\lambda \in [1, \infty)\) and let  \((\varphi_n)\)  be  the sequence of functions defined by the expression
\[
\varphi_n(t)=\frac{1}{2f_n(t/\lambda)}.
\]
Then \(\varphi_n \in C[0,1]\) and \(\|\varphi_n\|_\infty \leq 1.\) Consider the sequence  \((X_n)\) in \({\mathcal X}\)  defined by \((X_nf)(t)= \varphi_n(t)f(t/\lambda).\) Then \(\|X_n\| \leq 1\) and \((X_nf_n)(t) = 1/2\) for all \(t \in [0,1].\) Now, suppose that \(\lambda \in (0,1)\) and let  \((\varphi_n)\)  denote  the sequence of functions defined by the expression
\[
\varphi_n(t)= \left \{ \begin{array}{rl} \displaystyle{\frac{\lambda-t}{2  f_n(t/\lambda)},} & \text{if  } 0 \leq t < \lambda,\\
								0, & \text{if }\lambda \leq t  \leq 1.
		\end{array} \right. 
\]
Then \(\varphi_n \in C[0,1]\) and \(\|\varphi_n\|_\infty \leq 1.\) Consider the sequence  \((X_n)\) in \({\mathcal X}\)  defined by the expression 
\[
(X_nf)(t)= \left \{ \begin{array}{rl} \varphi_n(t)f(t/\lambda), & \text{if  }  0 \leq t \leq \lambda ,\\
								0, & \text{if } \lambda < t \leq 1,
		\end{array} \right.
\]
so that \(\|X_n\| \leq 1\) and \( (X_nf_n)(t) = \max \{0, (\lambda -t)/2\}\)  for all \(t \in [0,1]\) and all \(n \geq 1.\) In both cases we   conclude that  the family  \({\mathcal X}\) is a localizing subspace  of \({\mathcal B}(C[0,1]).\)

Next, we show that \({\mathcal X}\)  does not contain any nonzero compact operators. Take an operator \(X \in {\mathcal X}\backslash \{0\}\) and let \(\varphi \in C[0,1]\) be a nonzero function such that \(X\) is given by the expression \((\ast).\) Since \(\varphi\) is continuous and it does not vanish identically, and since \(\varphi(t)=0\) for all \(t > \lambda,\)  there is some \(\delta >0\) and there is an open interval \(I \subseteq [0,1]\) with \(\lambda I \subseteq [0,1]\) and such that \(|\varphi(t)| \geq \delta \) for all \(t \in \lambda I.\) Consider the infinite dimensional, closed subspaces 
\begin{eqnarray*}
E & = & \{ f \in C[0,1] \colon f(t)=0 \text{ for all } t \in [0,1] \backslash I\}, \\
F & = & \{ f \in C[0,1] \colon f(t)=0 \text{ for all } t \in [0,1] \backslash \lambda I\}.
\end{eqnarray*} 
Notice that  \(XE \subseteq F.\) We claim that the restriction \(X_{|E} \colon E \to F\) is onto, so that \(X\) cannot be compact. Indeed, let \(g \in F\) and consider the function defined by
\[
f(t)=  \left \{ \begin{array}{rl}  g(\lambda t)/\varphi (\lambda t), & \text{if } t \in I,\\
					 0, &  \text{if } t \in [0,1] \backslash I.
	\end{array} \right.
\]
It is easy to see that \(f \in E\) and \(g=Xf,\) as we wanted.
\end{proof}
\noindent
{\bf Acknowledgement.} This research was partially supported by  Ministerio de Educaci\'on, Cultura y Deporte  under Grant MTM 2009-08934. % and by Junta de Andaluc{\'\i}a under grants FQM-3737 and P09-FQM-4745.}

\end{document}